\documentclass[11pt]{article} 
\usepackage{authblk}
\usepackage{hyperref}
\newcommand*{\email}[1]{%
    \normalsize\href{mailto:#1}{#1}\par
    }
    
\usepackage[utf8]{inputenc} 
\usepackage{amssymb}
\usepackage{amsfonts}
\usepackage{amsmath}
\usepackage{mathtools}

\usepackage{geometry} 
\usepackage{graphicx} 

\usepackage{booktabs} 
\usepackage{array} 
\usepackage{paralist} 
\usepackage{verbatim} 
\usepackage{subfig} 

\usepackage{fancyhdr} 
\usepackage{sectsty}
\allsectionsfont{\sffamily\mdseries\upshape} 

\usepackage[nottoc,notlof,notlot]{tocbibind} 
\usepackage[titles,subfigure]{tocloft} 

\newtheorem{theorem}{Theorem}

\newtheorem{definition}[theorem]{Definition}

\newtheorem{lemma}[theorem]{Lemma}

\newtheorem{proposition}[theorem]{Proposition}
\newtheorem{remark}[theorem]{Remark}

\newcommand{\R}{\mathbb{R}}
\newcommand{\Z}{\mathbb{Z}}
\newcommand{\Q}{\mathbb{Q}}

\newenvironment{proof}[1][Proof]{\noindent\textbf{#1.} }{\ \rule{0.5em}{0.5em}}

\newcommand*{\transp}[2][-3mu]{\ensuremath{\mskip1mu\prescript{\smash{\mathrm t\mkern#1}}{}{\mathstrut#2}}}%

\begin{document}
\title{The Borsuk-Ulam theorem for  3-manifolds }

\author{Chahrazade Matmat}
\affil{  Département de mathématiques,
Université Fr\`eres Mentouri, Constantine 1, Algeria.
\email{c.matmat@yahoo.fr}}

\author{Christian Blanchet}
\affil{Université de Paris and Sorbonne Université, CNRS, IMJ-PRG, F-75006 Paris, France.
\email{Christian.Blanchet@imj-prg.fr}}


\maketitle
\begin{abstract}
We study the Borsuk-Ulam theorem for  triple $(M, \tau,\R^n)$, where $M$ is a compact, connected, $3$-manifold equipped with a fixed-point-free involution $\tau.$ The largest value of $n$ for which the Borsuk-Ulam theorem holds is called the $\Z_2$-index and in our case it takes value  $1$, $2$ or $3$.
 We fully discuss this index according to cohomological operations applied on
the characteristic class $x\in H^1(N, \Z_2)$, where $N=M/\tau$ is the orbit space.
 In oriented case, we obtain an expression of the index from the linking matrix of a surgery presentation of the orbit space.
We illustrate our results with examples, including a non orientable one.

\noindent
\textbf{2020 MSC} : 57K30, 57M60.\\
\textbf{Key words}: Borsuk-Ulam Theorem, 3-manifolds, surgery, linking forms.
\end{abstract}
\section*{Introduction}
Let $(X, \tau)$ be a free ${\Z}_2 $-space, which means that $X$ is a topological space and $\tau:X\rightarrow X$ is a fixed-point-free involutive homeomorphism. 
The Borsuk-Ulam theorem holds for a triple $(X, \tau, Y)$, where $Y$ is a topological space, if and only if for every continuous map $f: X\rightarrow Y$ there exists  $x\in X$ such that $f(x)=f(\tau(x))$. In this case, we say that the triple $(X, \tau, Y)$ has the Borsuk-Ulam property or $(X, \tau, Y)$ is a Borsuk-Ulam triple.
The original version of this theorem was conjectured by St. Ulam and proved by K. Borsuk  \cite{Borsuk} in $1933$ for the triple $(S^n, \tau, \R^n)$, where $\tau : S^n \rightarrow S^n$ is the antipodal map.
This result has  several generalizations and interesting applications see e.g. \cite {Jiri} .

 Keeping $X=S^n$, Corner and Floyd  \cite{Corner} proved the theorem for  $Y$  a finite dimensional differential manifold $M,$ Munkholm in \cite{Munkholm} proved that we can omit the condition of differentiability and assume that $Y$ is a compact topological m-manifold. On the other hand, Munkholm in \cite{Munkholm 1} and Minoru Nakaoka in \cite{Nakaoka} have respectively replaced the sphere $S^n$ with  closed topological (mod 2) $n$-homological spheres and (mod p) $n$-homological spheres, they assumed that $Y$ is a compact topological manifold and proved the Borsuk-Ulam theorem in these cases. The study of the Borsuk-Ulam theorem for triples $(X, \tau, \R^n)$, where $X$ is a low dimensional manifold is an interesting problem. Daciberg Lima Gonçalves  \cite{Daciberg 2} fully studied the case $(S, \tau, \R^2)$, where $S$ is a closed surface.

Our goal in this paper is to discuss the  Borsuk-Ulam index
 in the case where $M$ is a compact connected $3$-dimensional manifold. 
The result is known for certain families of $3$-manifolds : double covers of Seifert manifolds in \cite{BGHZ}, spherical manifolds in \cite{Daciberg 4}. General cohomological conditions are given in \cite{Daciberg 3}. Our purpose here is to review and reformulate these general conditions in terms of easily computable criterions. Our main original contribution concerns surgery presentations of oriented $3$-manifolds.

Before stating the results, let us introduce some notation. We denote by $N=M/\tau$ the orbit space.
Let  $x\in H^1(N, \Z_2)$ be the classifying class of  the principal $\Z_2$-bundle $M\twoheadrightarrow N$ i.e $x=\gamma^*(\alpha)$, where $\gamma:N\rightarrow \R P^\infty$  classifies the bundle  and $\alpha$ is the generator of $H^1(\R P^\infty, \Z_2)$. The class $x$ is not trivial since $M$ is connected.

Let $\beta : H^1( \ .\ , \Z_2) \longrightarrow H^2( \ .\ , \Z),$ $\beta_2 : H^1(\ . \ , \Z_2) \longrightarrow H^2(\ . , \Z_2)$ 
  be the Bockstein homomorphisms associated respectively  to the short exact sequences $0 \longrightarrow \Z \overset{\times 2}\longrightarrow \Z \longrightarrow \Z_2 \longrightarrow 0,$
  $0 \longrightarrow \Z_2 \overset{\times 2}\longrightarrow \Z_4 \longrightarrow \Z_2 \longrightarrow 0$.
The following well known theorem \cite{Daciberg 3,BGHZ} reduces the discussion to cohomological computations.

\begin{theorem} 
Let $M$ be a compact and connected $3$-dimensional manifold with fixed point free involution $\tau$, and corresponding classifying class $x\in H^1(N,\Z_2)$.
\begin{enumerate}
\item $1\leq ind_{\Z_2}(M, \tau)\leq 3$.
\item $ind_{\Z_2}(M, \tau)=1 \Longleftrightarrow \beta(x)=0 $.
\item $ind_{\Z_2}(M, \tau)= 3 \Longleftrightarrow  x^3\neq 0$.
\end{enumerate}
\label{th7}
\end{theorem}

 The triple cup used in the third equivalence can be tricky to compute in examples.
 We will reformulate this criterion. The class $x\in H^1(N,\Z_2)$ can be used  to define homology and cohomology with twisted coefficients. 

 We use the notation $\Lambda^-$ for the coefficient ring $\Lambda$ when the homology or cohomology is twisted with the representation
 $\phi: \pi_1(N) \rightarrow \mathrm{Aut}(\Lambda)$, where
 $\phi([\gamma])(\lambda)=(-1)^{<x,\gamma>}\lambda$. 
  Let  $\beta_2^- : H^2(\ . \ , \Z_2) \longrightarrow H^3(\ . \ , \Z_2)$ be the Bockstein homomorphism associated
 to the short exact sequence $0 \longrightarrow \Z_2 \overset{\times 2}\longrightarrow \Z_4^- \longrightarrow \Z_2 \longrightarrow 0$.

\begin{theorem} \label{th_reformulate3}Let $M$ be a compact and  connected  $3$-dimensional manifold with fixed point free involution $\tau$, and classifying class $x\in H^1(N,\Z_2)$.

Let  $\bar x\in H_1 (N, \Z_2)$ be the Poincaré dual of $\beta_2 (x)$. Then the following are equivalent
\begin{enumerate}[(i)]
\item  $ind_{\Z_2}(M, \tau)= 3$,
\item  $\langle x,\bar x\rangle \neq 0$,
\item $(\beta_2^-\circ \beta_2)(x)\neq 0$.
\end{enumerate}
\end{theorem}

In oriented case, a formula from Turaev \cite{Turaev} gives the triple cup in terms of the linking pairing, which implies an easier criterion stated in Theorem \ref{th_linking}.
Based on this formula, our  main contribution is a general result for oriented $3$-manifolds using a surgery presentation of the orbit space. Recall that any oriented compact $3$-manifold can be given such a surgery presentation.
\begin{theorem} 
Let $(M,\tau)$ be a compact oriented $3$-manifold with oriented free involution $\tau$. Suppose that the quotient $M/\tau$
is homeomorphic to $N_{\mathcal{L}}=\partial W_{\mathcal{L}}$, the result of surgery along an $m$ components framed link $\mathcal{L}$ in $3$-sphere, with associated linking matrix $B_\mathcal{L}$. Let \mbox{$x\in H^1(N_{\mathcal{L}},\Z_2)\cong\ker(B_{\mathcal{L}}\otimes \Z_2)$} be the classifying class, and
let \mbox{$X\in H^2(W_\mathcal{L},N_\mathcal{L},\Z)\cong \Z^m$} be an integral lift of the coboundary $\delta(x)\in H^2(W_{\mathcal{L}},N_\mathcal{L},\Z_2)\cong \Z_2^m$, then we have:
\begin{enumerate}
\item 
$ind_{\Z_2}(M, \tau)=1$ if and only if $\frac{1}{2}B_\mathcal{L}.X$  vanishes in $ \mathrm{coker}(B_\mathcal{L})\approx H^2(N_\mathcal{L},\Z)$.
\item $ind_{\Z_2}(M, \tau)=3 \Longleftrightarrow \frac{1}{2}\transp{X}.B_\mathcal{L}.X\neq 0 \mathrm{\ mod\ }2$.
\end{enumerate}
\label{th8}
\end{theorem}

We apply our results to a few families of example. As a warm up we first consider the double cover of lens spaces. Since lens spaces are Seifert fibrations, the result is contained (although hidden) in \cite{BGHZ}, case $(0; (o_1,0); (q,p))$ in Orlik notation. We then fully discuss double covers of mapping tori. We consider the case of surgery presentation on  algebraically split links. We finally study all free involutions on $S^1 \times S^2$, which include a non oriented one. In Proposition \ref {prop 22}, we prove the Borsuk-Ulam theorem for the $3$-Klein bottle $K^3$ with natural involution.

This paper is divided into four sections.  Section 1 is reviewing of the Borsuk-Ulam theorem and the $\Z_2$-index. 
 In Section 2 we specialise to $3$-dimensional manifolds and complete the known results, including a proof
 of Theorem \ref{th_reformulate3} stated in Introduction.
In Section 3 we prove Theorem \ref{th8} in which we give criterions for oriented $3$-manifolds given by surgery presentation.
The last Section is devoted to study examples and applications. 
\section{Review of the Borsuk-Ulam theorem}

In this Section, we start by recalling some facts and results related with the Borsuk-Ulam property.
\begin{definition}\cite{Jiri}
Let $(X_1, \tau_1)$, $(X_2, \tau_2)$ be two free $\Z_2$-spaces. A continuous map $f: X_1 \rightarrow X_2$ is said to be $\Z_2$-equivariant if it commutes with the $\Z_2$-actions i.e for all $x\in X$, we have $f(\tau_1(x))=\tau_2(f(x))$. 
\end{definition}
This can be expressed by the following commutative diagram:
\begin{equation*}
\begin{array}{ccc}
    X_1 & \overset{f} \longrightarrow&  X_2 \\
    \tau_1\downarrow &  & \downarrow \tau_2 \\
 X_1 & \overset{f} \longrightarrow & X_2 \\
  \end{array}
\end{equation*}
In this case, we write $f: X_1 \overset{\Z_2}\rightarrow X_2$.
\begin{definition}
Let $(X, \tau)$ be a free $\Z_2$-space. The $\Z_2$-index is defined as $$ind_{\Z_2}(X, \tau)=min\{n\in\{0, 1, 2,....\infty\}/ \exists f : X\overset{\Z_2}\rightarrow S^n\}$$
Here, we take the standard antipodal $\Z_2$-action on $S^n$.
\end{definition}

The $\Z_2$-index may be a natural number or $\infty$; it depends on the space $X$ and on the $\Z_2$-action $\tau$.
 The validity of the Borsuk-Ulam theorem
for $(X, \tau, \R^n)$ 
 is equivalent to the non existence of an equivariant map $f: X \rightarrow S^{n-1}$ (see e.g. \cite[Proposition 2.2]{Daciberg 3}). This is also equivalent to the non existence of a reduction to $\R P^{n-1}$ of the classifying map $X/\tau\rightarrow \R P^\infty$. This shows that the $\Z_2$-index is an homotopy problem which can be studed through obstruction theory.
 It follows that for a 
 connected finite dimensional  CW-complex  $X$  with involution $\tau$, the $\Z_2$-index  is finite and  bounded by one and the dimension of $X$.
 
The lower case is discussed in \cite[Theorems 3.1]{Daciberg 3}. Using the Bockstein exact sequence, the result can be stated as follows.
\begin{theorem}
Let $(X, \tau)$ be a CW-complex with involution and $x\in H^1(X/\tau,\Z_2)$ be
 the classifying map of the principal bundle
$X \rightarrow X/\tau$. 
 The index $ind_{\Z_2}(X, \tau)$ is equal to one if and only if $\beta(x)=0$, where $\beta : H^1(X/\tau,\Z_2)\rightarrow H^2(X/\tau,\Z)$ is the Bockstein homomorphism.
\end{theorem}

For an $m$-dimensional manifold, the upper case is also well known\cite[Theorems 3.4]{Daciberg 3}.

\begin{theorem}
Let $(M, \tau)$ be an $m$-dimensional manifold with involution and $x\in H^1(M/\tau,\Z_2)$ be
 the classifying map of the principal bundle
$M \rightarrow M/\tau$. 
 The index $ind_{\Z_2}(X, \tau)$ is equal to $m$ if and only if the $m$-fold cup product $x^m$ does not vanish.
\end{theorem} 
\section{The Borsuk-Ulam theorem for $3$-manifolds}

Here we specialize the statements of the previous section to the case where $X $ is a connected $3$-dimensional manifold with free involution $\tau$.
We immediatly obtain Theorem \ref{th7} stated in introduction.

The following lemma will allow the reformulation of the index 3 case.
\begin{lemma}Let $Y$ be a CW-complex, and $\beta_2 : H^1(Y, \Z_2) \longrightarrow H^2(Y, \Z_2)$ be the Bockstein homomorphism associated to the short exact sequence $0 \longrightarrow \Z_2 \overset{\times 2}\longrightarrow \Z_4 \longrightarrow \Z_2 \longrightarrow 0.$ \\
a) 
For every $x\in H^1(Y, \Z_2)$, we have $x\smallsmile x=\beta_2(x)$. \\
b) Let   $\beta^-_2 : H^2(Y, \Z_2) \longrightarrow H^3(Y, \Z_2)$ be the Bockstein homomorphism associated with the short exact sequence $0 \longrightarrow \Z_2 \overset{\times 2}\longrightarrow \Z_4^- \longrightarrow \Z_2 \longrightarrow 0$ for the $x$-twisted coefficients $\Z_4^-$. Then
for  $x\in H^1(Y, \Z_2)$, we have $x\smallsmile x \smallsmile x=(\beta_2^-\circ \beta_2)(x)$.
\label{lemma 7}
\end{lemma}

\begin{proof}

The cohomology class $x\in H^1(Y, \Z_2)$ can be represented by a map
$f: Y\rightarrow \R P^\infty$, which means $x=f^*(\alpha)$ where $\alpha$ is the generator of $H^1(\R P^\infty, \Z_2)$.
By functoriality, it is enough to prove the formulas for the generator $\alpha$. Recall that $\alpha$ freely generates the ring $H^*(\R P^\infty,\Z_2)$. The standard cellular structure for $\R P^\infty$ has one cell $e_k$ in each dimension $k$, with boundary $\partial e_{2j}=2e_{2j-1}$,
$\partial e_{2j+1}=0$. The coefficient map $\Z_2\overset{\times 2}\rightarrow \Z_4$ induces an isomorphism  $H^1(\R P^\infty,\Z_2)\overset{\sim}\rightarrow H^1(\R P^\infty,\Z_4)=\Z_2$. From the Bockstein exact sequence, we get that $\beta_2: \Z_2=H^1(\R P^\infty,\Z_2)\rightarrow  H^2(\R P^\infty,\Z_2)=\Z_2$ is an isomorphism, which proves a).

For proving b), we need to understand the cohomology of $\R P^\infty$ with $\Z_2$-twisted coefficients. For this purpose, we use a deck equivariant cell decomposition of the universal cover $S^\infty$ with two cells in each dimension $k$, $e_k'$ and $e_k''$. The deck transformation is $-id_{S^\infty}$. It respects orientation on odd dimensional cells and changes the orientation on even dimensional cells. The twisted complex is generated by
the cells $e'_k$ identified with $-e''_k$, with boundary
 $\partial^- e'_{2j}=0$,
$\partial^- e'_{2j+1}=2e'_{2j}$. The coefficient map $\Z_4^-\rightarrow \Z_2$ induces an isomorphism  $\Z_2=H^1(\R P^\infty,\Z_4^-)\overset{\sim}\rightarrow H^1(\R P^\infty,\Z_2)$. In the Bockstein exact sequence, we get that 
$\beta^-_2: \Z_2=H^1(\R P^\infty,\Z_2)\rightarrow  H^2(\R P^\infty,\Z_2)=\Z_2$ vanishes, then
the coefficient map $\Z_2\overset{\times 2}\rightarrow \Z_4^-$ induces an isomorphism  $H^2(\R P^\infty,\Z_2)\overset{\sim}\rightarrow H^2(\R P^\infty,\Z_4^-)=\Z_2$, finally
$\beta^-_2: \Z_2=H^2(\R P^\infty,\Z_2)\rightarrow  H^3(\R P^\infty,\Z_2)=\Z_2$ is an isomorphism. Since $H^3(\R P^\infty,\Z_2)=\Z_2$ is generated by $\alpha^3$, we have $\alpha^3=(\beta_2^-\circ \beta_2)(\alpha)$ which proves b) by functoriality.
\end{proof}

\vspace{10pt}\begin{proof}[Proof of Theorem \ref{th_reformulate3}]

Formula b) in previous lemma establishes equivalence $(i)\Leftrightarrow (iii)$. The computation below shows equivalence with $(ii)$.

$$<x\smallsmile x \smallsmile x,[N]>=<x\smallsmile \beta_2(x),[N]>=
<x,\beta_2(x) \smallfrown [N]>=<x,\overline x>\ .$$
\end{proof}

In oriented case, a formula from Turaev \cite{Turaev} gives the triple cup in terms of the linking pairing, which gives an easier criterion.

\begin{theorem} \label{th_linking}Let $M$ be a compact oriented and  connected  $3$-dimensional manifold with fixed point free oriented  involution $\tau$, and classifying class $x\in H^1(N,\Z_2)$.
Let  \mbox{$\widetilde x\in Tors(H_1 (N, \Z))$} be the Poincaré dual of $\beta (x)$.
Denote by $\mathcal{L}_N:Tors( H_1 (N, \Z)) \otimes Tors( H_1 (N, \Z)) \longrightarrow \Q/\Z$ the linking pairing of $N$. Then we have
 $$ind_{\Z_2}(M, \tau)=3 \Longleftrightarrow  \mathcal{L}_N(\widetilde{x}, \widetilde{x})\neq 0\ .$$
\label {th2}
\end{theorem}
\vspace{10pt}\begin{proof}[Proof of Theorem \ref{th2}]

Turaev's theorem I in \cite{Turaev}, allows to express evaluation of the triple cup product on the fundamental class.
We reproduce below the argument in our specific case. Let $Dx\in H_2(N,\Z_2)$ be the Poincaré dual of $x$. Denote the  coefficient homomorphisms as follows:
$$\phi : H_1(\ .\ ; \Z) \rightarrow H_1(\ .\ ; \Z_2)\ ,\   \psi : H_2(\ .\ ; \Z_2) \rightarrow H_2(\ .\ ; \Q/\Z)\ .$$
The definition of the linking pairing uses the Bockstein homomorphism $$B: H_2(N,\Q/\Z)\rightarrow H_1(N,\Z)\ .$$
By functoriality of the Bockstein exact sequence with respect to coefficients, we get $\overline x=\phi(\tilde x)$,
and $B(\psi(Dx))=\tilde x$.
Then we have
$$\mathcal{L}_N(\tilde x,\tilde x)=\psi(Dx).\tilde x =\frac{1}{2}Dx.\phi(\tilde x)=\frac{1}{2}\langle x,\overline x\rangle\ .$$
The result follows.

\end{proof}
\section{Computation for surgery presentation of oriented 3-manifolds}

By a theorem of Lickorish and Wallace (see theorem 2.1 in \cite{Nikolai}), any compact oriented $3$-manifold can be obtained  by surgery on a framed link in $S^3$. In this section we discuss the $\Z_2$-index for an oriented $3$-manifold $M$ equipped with oriented involution $\tau$ when the quotient $N=M/\tau$ is given by a surgery presentation,
 which allows  to decide when the triple $(M, \tau, \R^n)$ has the Borsuk-Ulam property. 
 
We first recall some facts about surgery presentations. 
Let $\mathcal{L}$ be a framed link in the sphere $S^3$, $ N_\mathcal{L}$ be the oriented, compact, connected  3-manifold obtained by Dehn surgery on $\mathcal{L}$.
Recall that, if $L_1, L_2, ......, L_m$ are the components of $\mathcal{L}$, the linking matrix of $\mathcal{L}$ is an $m\times m$ matrix of integers $a_{ij}$, such that $a_{ij}=lk(L_i, L_j)$ if $i\neq j$, here $lk(L_i, L_j)$ denote the linking number of  $ L_i$ and $L_j$, and $a_{ii}$ is the framing of $L_i$. Denote by  $B_\mathcal{L}$ the linking matrix of $\mathcal{L}$, and by $\bar{B}_\mathcal{L}$ its reduction modulo $2$.
It is known (see for example \cite{Nikolai}) that $\mathcal{L}$ defines
a compact $4$-manifold  $W_{\mathcal{L}}$ such that $\partial W_{\mathcal{L}} =N_\mathcal{L}$. 
In fact, $W_{\mathcal{L}}$ is obtained by attaching $m$ index $2$ handles $D^2\times D^2$ to the ball $D^4$ via an oriented embedding $\amalg_m (- S^1\times D^2) \hookrightarrow S^3$.
 If $A$ is one of the groups : $\Z, \Z/n\Z$ or $\Q/\Z$, then  
 \begin{center}$H_2(W_{\mathcal{L}}, A)\simeq H^2(W_{\mathcal{L}}, A) \simeq A^m$, $H_2(W_{\mathcal{L}}, N_\mathcal{L}, A)\simeq H^2(W_{\mathcal{L}}, N_\mathcal{L}, A)\simeq A^m$
 \end{center}
 A canonical basis for homology is represented by the cores of the handles. The dual basis is used for cohomology. In the exact sequences
\begin{equation}0 \longrightarrow H_2(N_{\mathcal{L}}, A)\overset{i_*}\longrightarrow H_2(W_{\mathcal{L}}, A) \overset{j_*}\longrightarrow H_2(W_{\mathcal{L}}, N_\mathcal{L}, A)\overset{\partial}\longrightarrow H_1(N_{\mathcal{L}}, A)\longrightarrow 0  \label{sq(W)} \end{equation}
in homology and respectively
\begin{equation}0 \longrightarrow H^1(N_{\mathcal{L}}, A) \overset{\delta}\longrightarrow H^2(W_{\mathcal{L}}, N_\mathcal{L}, A) \overset{j^*}\longrightarrow H^2(W_{\mathcal{L}}, A)\overset{i^*}\longrightarrow H^2(N_{\mathcal{L}}, A)\longrightarrow 0  \label{sq(W_1)} \end{equation}
in cohomology, 
the homomorphism $j_*$ (respectively $j^*$ ) is given by the linking matrix $B_\mathcal{L}$ of $\mathcal{L}$. We  always choose these basis for $H_2(W_{\mathcal{L}}, A)$ (respectively $ H^2(W_{\mathcal{L}},N_\mathcal{L}, A)$) and $H_2(W_{\mathcal{L}}, N_\mathcal{L}, A)$  (respectively $H^2(W_{\mathcal{L}},  A)$). In addition to this, the exact sequence  \eqref{sq(W)}  induces the isomorphisms 
\begin{center}\mbox{$\ker B_\mathcal{L}\cong H_2(N_{\mathcal L}, A)\cong H^1(N_{\mathcal L}, A)$} and \mbox{$\mathrm{coker} B_\mathcal{L}\cong H_1(N_{\mathcal L}, A)\cong H^2(N_{\mathcal L}, A)$}.\end{center}
\begin{proof}[Proof of Theorem \ref{th8}]

Let $(M,\tau)$ be a compact oriented $3$-manifold with oriented free involution $\tau$. The quotient $N=M/\tau$
is homeomorphic to  $N_{\mathcal{L}}=\partial W_{\mathcal{L}}$, where $W_{\mathcal{L}}$ 
 is obtained from $D^4$ by attaching $m$ index $2$ handles 
along the framed link $\mathcal{L}$, as explained above.
 Let $x\in H^1(N_{\mathcal{L}},\Z_2)$ be the classifying class, and
let \mbox{$X\in H^2(W_{\mathcal{L}},N_{\mathcal{L}},\Z)\cong \Z^m$} be an integral lift  of the coboundary $\delta(x)\in H^2(W_{\mathcal{L}},N_{\mathcal{L}},\Z_2)$.
\begin{enumerate}
\item
We will get the first equivalence from Theorem \ref{th7} by computing $\beta(x)$.
Recall that the Bockstein homomorphism $\beta: H^1(N_{\mathcal{L}},\Z_2)\rightarrow H^2(N_{\mathcal{L}},\Z)$ is the connecting homomorphism associated with
the short exact sequence of cochain complexes
$$0\longrightarrow C^*(N_{\mathcal{L}},\Z)\overset{\times 2}\longrightarrow C^*(N_{\mathcal{L}},\Z)\overset{\rho}\longrightarrow C^*(N_{\mathcal{L}},\Z_2)\rightarrow 0\ .$$
Here we may use the complexes associated with any cell decomposition. Functoriality of the Bockstein exact sequence asserts that the Bockstein homomorphism can be computed with another short exact sequence of cochain complexes provided there exists a chain map to or from the previous one, inducing cohomology isomorphisms.
 The following commutative diagram will provide such.
\begin{equation}
  \begin{array}{ccccccccc}
       0 \rightarrow& H^1(N_{\mathcal{L}}, \Z) & \overset{\delta}\longrightarrow & H^2(W_{\mathcal{L}},N_{\mathcal{L}},\Z) & \overset{B_{\mathcal{L}}}\longrightarrow & H^2(W_{\mathcal{L}},\Z) &\overset{i^*} \longrightarrow& H^2(N_{\mathcal{L}}, \Z) & \rightarrow 0 \\
     &\times 2 \downarrow &  & \times 2\downarrow &  & \times 2\downarrow &  & \times 2\downarrow &  \\
   0 \rightarrow & H^1(N_{\mathcal{L}}, \Z) & \overset{\delta}\longrightarrow & H^2(W_{\mathcal{L}},N_{\mathcal{L}},\Z) & \overset{B_{\mathcal{L}}}\longrightarrow& H^2(W_{\mathcal{L}},\Z) & \overset{i^*}\longrightarrow & H^2(N_{\mathcal{L}}, \Z) & \rightarrow 0 \\
    & \rho\downarrow&  & \rho\downarrow&  & \rho\downarrow &  & \rho\downarrow &  \\
    0\rightarrow & H^1(N_{\mathcal{L}}, \Z_2) & \overset{\delta}\longrightarrow & H^2(W_{\mathcal{L}},N_{\mathcal{L}},\Z_2) & \overset{\bar{B}_{\mathcal{L}}}\longrightarrow & H^2(W_{\mathcal{L}},\Z_2)& \overset{i^*}\longrightarrow & H^2(N_{\mathcal{L}}, \Z_2)& \rightarrow 0 \\
  \end{array}
\label{diagr 3}
\end{equation}
For $A=\Z$ or $A=\Z_2$, the cohomology $H^*(N_{\mathcal{L}}, A)$ can be computed from the cochain complex
$$C^0_A=A\overset{0}\longleftarrow  C^1_A=H^2(W_{\mathcal{L}},A) \overset{B_{\mathcal{L}}\otimes A} \longleftarrow  C^2_A=H^2(W_{\mathcal{L}},N_{\mathcal{L}},A)\overset{0}\longleftarrow C^3_A=A$$
Let us denote by $V_\mathcal {L}$ an open tubular neighborhood of $\mathcal{L}$ in $S^3$. Then we have $N_{\mathcal{L}}=(S^3-V_\mathcal{L})\cup\left(\amalg_m D^2\times S^1\right)$ and we will consider $\dot N_{\mathcal{L}}=S^3-V_\mathcal{L}$ as included in $N_{\mathcal{L}}$.
Assume that the cell structure is such that
\begin{itemize}
 \item $\dot N_{\mathcal{L}}$ is a subcomplex, 
\item  it has one $0$-cell,
  \item  there is a relative cell structure for $(N_{\mathcal{L}},\dot N_{\mathcal{L}})$ with one $2$-cell $e^2_i$ for each component of $L$ which is the image of the oriented disc $ D^2\times 1$, completed with $m$  $3$-cells.
   \end{itemize}
A chain map to the cell cochain complex is defined as follows:
\begin{itemize}
\item the generator in degree zero evaluates $1$ on the $0$-cell,
\item the $i$-th generator in degree $1$ evaluates on a cell $\gamma$ as the linking $lk(L_i,\gamma)$,
\item the $i$-th generator in degree $2$ evaluates $1$ on $e^2_i$ and $0$ on the other $2$-cells,
\item the generator in degree $3$ evaluates $1$ on each oriented $3$-cell.
\end{itemize}
The connecting homomorphism $\beta$ is  obtained as follows.
 We have \mbox{$B_{\mathcal{L}}X\in Ker\rho$} and there exists $Y \in H^2(W_{\mathcal{L}},N_{\mathcal{L}},\Z)\cong\Z^m$ such that $2Y=B_{\mathcal{L}}.X$. Then $Y$ represents $\beta(x)=i^*(Y)\in H^2(N_{\mathcal{L}},\Z)\cong\mathrm{coker}(B_\mathcal{L})$.
 Finally $\beta (x)=i^*(\frac{1}{2}B_\mathcal{L}.X)$, so by 1) of Theorem \ref{th7}, we get the first equivalence of Theorem \ref{th8}.
\item
To show that $ind_{\Z_2}(M, \tau)=3 \Longleftrightarrow \frac{1}{2}\transp{X} .B_\mathcal{L}.X\neq 0$ mod $2$, we will use Theorems \ref{th_reformulate3} and \ref{th2}. Let $\bar x\in Tors H_1(N_{\mathcal{L}}, \Z_2)$ be the Poincaré dual of $\beta_2(x)$ and $\tilde{x} \in Tors H_1(N_{\mathcal{L}}, \Z)$ the Poincaré dual of $\beta(x)$.
We have $$\bar{x}=(D\circ \beta_2)(x)=(D\circ \rho \circ i^* )(Y) \; \text{and} \;\tilde{x} =D \circ \beta(x) = D \circ i^*(Y)$$
 where $D$ denotes the Poincaré duality isomorphism
 and \mbox{$Y=\frac{1}{2}B_{\mathcal{L}}(X)$}. Then 
$$\langle x,\bar{x}\rangle=2 \mathcal{L}_N(\tilde x,\tilde x)=2\langle x, (D\circ i^*)(Y)\rangle=2\langle \delta x, Y\rangle .$$
Using that $\delta (x) =\rho(X)$, we deduce
\begin{eqnarray*}
\langle x,\bar{x}\rangle&=&\transp{X}Y \text{ mod. $2$}\\
&=&\frac{1}{2} \transp{X} B_\mathcal{L}X \text{ mod. $2$\ .}
\end{eqnarray*}
The result follows.
\end{enumerate}
\hfill \end{proof}

\noindent{\bf Recovering classical results.} We may quickly check that Theorem \ref{th8} recovers well known results. For the sphere $S^3$ with antipodal action the orbit space is the projective space $\R P^3$, obtained by surgery on an unknot with framing 2.  We have $B_{\mathcal{L}}=(2)$. Let $x$ be the classifying class in $H^1(\R P^3, \Z_2) \simeq \Z_2$. An integral lift of $\delta x$ is $1$, and 
$\frac{1}{2}\transp{X} B_\mathcal{L}X =1$. This proves that the index is $3$ as expected.

For the projective space $\R P^3$ with the action $\tau$ induced by the multiplication by the complex number $i $, the orbit space is the lens space 
 $L(4, 1)$, which is obtained by surgery on a $(-4)$-framed unknot
\cite[Example 5.3.2]{Gompf}, and $B_\mathcal{L}=(-4)$. The classifying class $x$ is the generator
of $H^1(L(4, 1),\Z_2)=\Z_2$. An integral lift of $\delta (x)$ is $1$.
We have 
$\frac{1}{2} B_{\mathcal{L}} X=2 $ does not vanish in $H^2(L(4, 1),\Z)=\Z_4$,
and
$\frac{1}{2} \transp{X} B_\mathcal{L}X \text{ mod. 2\: }= - 2=0 \text{ mod }2$.
{This proves $ind_{\Z_2 }(\R P^3, \tau)\leq2$. Moreover,
$\frac{1}{2} B_{\mathcal{L}} X=2$  does not vanish in  \mbox{$\mathrm{coker}( B_{\mathcal{L}})= \Z_4$}. Therefore,  $ind_{\Z_2 }(\R P^3, \tau)=2.$}
This result was first obtained by Stephan Stolz 
\cite{stolz}\footnote{The index there is called level and is different from our definition, namely the level $s(X,\tau)$  is $ind_{\Z_2 }(X, \tau)+1$.}
 in 1989.
\section{Some examples and applications}
\subsection{Application to  lens spaces}
{Lens spaces are classical examples of closed orientable 3-manifolds. They play an important role in the history of algebraic topology.
Their classical definition was stated first by Tieze in 1908 and their name "Lens spaces" was introduced by Threlfall and Seifert in 1933.
There are many descriptions of Lens spaces (see \cite{Rolfsen}), the first one is to consider them as the quotient of the 3-ball, where the top hemi-sphere is identified with the bottom hemi-sphere by a rotation of angle $2\pi q/p$ about the Z-axis followed by a reflexion in $(x-y)$ plane, for some $p\in \Z$,  $p\geq 2$ and some $q\in \Z$ relatively prime to $p$.
The 3-ball used in this definition is often drawn in the shape of a lens. In the next description, we are going to show that it represents a fundamental domain of a $\Z/p$ action on the 3-sphere $S^3$ as follows.}

Let $p, q$ be relatively prime integers such that $p\geq 2$ and $1<q < p$. Consider the transformation 
\begin{equation*}
  \begin{array}{ccc}
   T: \Z/p \times S^3& \longrightarrow & S^3 \\
    ([k], (z_1, z_2))& \longmapsto &  (z_1 e^{\frac{2\pi ik}{p}}, z_2 e^{\frac{2\pi ikq}{p}})\\
 \end{array}
\end{equation*}
$T$  generates an action of the group $\Z/p$ in $S^3$. These action is free and finite, so the projection $\pi : S^3 \longrightarrow S^3/T$ is a covering map with $p$ sheets. The orbit space $S^3/T$ is called a Lens space and denoted by $L(p, q).$\\
The purpose of this section is to study the Borsuk-Ulam theorem for double covering of Lens spaces using surgery presentation. 
We Know that they can be obtained from the trivial knot by a rational surgery with framing $-\frac{p}{q}$, see for example \cite{Nikolai} and \cite{Rolfsen}. But it can also be represented by an integral surgery on a framed link $\mathcal{L}$ having $n$-components with framing $(a_i)_{i=\overline{1,n}}$ (see the same references or \cite{Kirby}), where $-\frac{p}{q}=[a_1, a_2, ...., a_n]$ is a continued fraction decomposition.
That is a decomposition of the form $$-\frac{p}{q}=a_1 -\frac{1}{a_2-\frac{1}{\ddots\; a_{n-1}-\frac{1}{a_n}}}.$$
The linking matrix has diagonal entries $a_i$. The sign of the non zero linking numbers depends on a choice of orientations. It is convenient here and in the next example to follows \cite[Figures 14,17]{Kirby}.
Then we get  the $n\times n$ matrix defined by
 $$B_{\mathcal{L}}=\left(
  \begin{array}{ccccc}
  a_1 & -1 & 0 & \cdots& 0 \\
 -1 & a_2& -1& \ddots & 0 \\
 0 & \ddots & \ddots &\ddots &  0 \\
  \vdots  & \ddots &  -1 &a_{n-1}  &-1 \\
 0  &\cdots & 0& -1 &a_n\\
  \end{array}
\right)$$ 
Note that $\mathcal{L}$  also specifies a 4-manifold $W_\mathcal{L}$ with boundary $L(p,q)$, obtained by adding
2-handles to the 4-ball along $\mathcal{L}$. Using the matrix above, homology and cohomology groups of lens spaces can be computed. In a part of the proof of Proposition \ref{prop L} below, we obtain that 
 $$ H_2(L(p, q), \Z)\simeq 0 \simeq H^1( L(p, q), \Z), \; \; \; H_1(L(p,q), \Z) \simeq \Z_p \simeq H^2 (L(p, q), \Z)$$  and also
$$H^{1}(L(p, q), \Z_2) \simeq
\left\{
\begin{array}{ccc}
    \Z_2  & \text{ if}\; p \; \text{is even}   \\
   0  & \text{if}\; p\;  \text{is odd}
\end{array}
\right. $$
After this introduction, we see that a Lens space $L(p,q)$ has connected double cover unique up to equivalence, if and only if
$p$ is even. Indeed, we have
$Hom(\pi_1(N), \Z_2)\approx Hom (H_1(N), \Z_2)\approx H^1 (N, \Z_2)$. From now we assume that $p$ is even, then 
the non trivial element $\alpha \in H^{1}(L(p, q), \Z_2)$ gives rise of a pair $(M, \tau)$, where $M$ is a closed connected 3-manifold and $\tau$ is a fixed point free involution on $M$ associated to the double covering with  action $\tau$ given by the non trivial deck transformation. We can check that the covering space $M$ is itself a lens space, namely $L(\frac{p}{2},q)$.
\begin{proposition} 
Let $L(p, q)$ be a lens space such that $p$ is even, $(M, \tau)=(L(\frac{p}{2},q),\tau)$ be as above, and $x\in H^1(L(p,q), \Z_2)$ be the non trivial class (which classifies the cover).
\begin{enumerate}
\item We have $\beta(x)\neq 0$. Therefore $ind_{\Z_2}(M, \tau)\geq2$ and $(M, \tau, \R^2)$ is always a Borsuk-Ulam triple.
\item  $ind_{\Z_2}(M, \tau)=3$ if and only if  $p\equiv 2$ mod 4.
\end{enumerate}
\label{prop L}
\end{proposition}
\begin{proof}
We are going to use the results of Theorem \ref{th8}. By induction we see that the linking matrix is equivalent to the matrix \\
$$B_{\mathcal{L}}=\left(
  \begin{array}{cccccccccc}
    0 & 0 & 0  & \ldots  &0 &  \alpha_1\\
    -1& 0 &0&\ldots &  0 &  \alpha_2  \\
  0   & -1  & 0 &\ldots &0 &  \alpha_3 & \\
\vdots & \ddots &\ddots  &  &    & \vdots  \\
    0 & \ldots &   0&  -1&   0&  \alpha_{n-1} \\
    0 & \ldots &   0&  0& -1 &\alpha_n= a_n
  \end{array}
\right)
$$
where $\alpha_n =a_n$, $\alpha_{n-1} =-1+a_{n-1}a_n$ and $\alpha_i =-\alpha_{i+2}+a_i \alpha_{i+1}$ for $n-2\geq i\geq 1$. Furthermore, we remark that $$\frac{\alpha_1}{\alpha_2}=\frac{a_1\alpha_2 -\alpha_3}{\alpha_2}= a_1-\frac{\alpha_3}{\alpha_2}= a_1-\frac{1}{\frac{\alpha_2}{\alpha_3}}=a_1-\frac{1}{a_2 -\frac{1}{a_3-\frac{\alpha_4}{\alpha_3}}}= \dots=a_1 -\frac{1}{a_2-\frac{1}{\ddots\; a_{n-1}-\frac{1}{a_n}}}.$$
Hence we have $\frac{\alpha_1}{\alpha_2} = -\frac{p}{q}$.
For each $i$ we get $\alpha_i$ coprime with $\alpha_{i+1} $. We deduce \mbox{$det( B_\mathcal{L})=\alpha_1=\pm p$}. Therefore, as expected, we get $$\ker (B_\mathcal{L}) \simeq H_2(L(p, q), \Z)\simeq H^1(L(p, q), \Z)=0$$ $$coker (B_\mathcal{L})= H_1(L(p,q), \Z) \simeq \Z_p \simeq H^2 (L(p, q), \Z)$$ and $$ ker (B_\mathcal{L}\otimes \Z_2)\simeq H_2(L(p, q), \Z_2)\simeq  H^1(L(p, q), \Z_2) \simeq \left\{
\begin{array}{ccc}
    \Z_2   & \mbox{\text if} \:\: \text{ p is even }\\
   0  & \mbox{\text if} \: \:\text{ p is odd}
 \end{array}
\right..$$
Consider the case where $p$ is even and let $x\in H^1(L(p,q), \Z_2)\simeq \Z_2$ be the classifying class which is the unique non trivial element in this group. Using the commutative diagram \ref{diagr 3}, consider $X\in H^2( W_{\mathcal{L}}, N_{\mathcal{L}}, \Z) \simeq \Z^n$ we can choose as  integral lift of $\delta(x)$ the vector $X=\transp(\alpha_2, \alpha_3, ...,  \alpha_{n-1}, \alpha_n, 1)$. 
Therefore,  $$i^*(\frac{1}{2} B_{\mathcal{L}} X )=\frac{1}{2} \transp(\alpha_1, 0,.......,0) \neq 0 \in H^2( N_\mathcal{L}, \Z) \cong \Z_p,$$ Hence from statement 1 in \ref{th8} we obtain $ind_{\Z_2}(M, \tau)\geq2$.\\ 
Furthermore, $$\langle x,\bar{x}\rangle = \frac{1}{2} \transp X B_{\mathcal{L}} X \;\text{mod}\; 2= \frac{1}{2} \alpha_1\alpha_2\text{ mod }2=\frac{1}{2} pq\text{ mod }2.$$
From statement 2 in \ref{th8}, this proves that $ind_{\Z_2}(M, \tau)=3$ if and only if  $p\equiv 2$ mod 4.
\end{proof}

\subsection{Double covers of torus bundles.}
In this part we apply our surgery method for Borsuk-Ulam index to double covers of oriented torus bundles over the circle. Following Thurston classification, these oriented $3$-manifolds split in three subclasses having respectively Euclidean geometry, Nil geometry and Sol geometry. As far as we know, the discussion of Borsuk-Ulam index in the last case, i.e. when the monodromy is Anosov, is knew.

\begin{definition}

Let $A$ be a self diffeomorphism  of the torus $T^2$.
A torus bundle is the identification space
$$T_A= T^2 \times [0,1]/ (x,1) \sim (A(x), 0)$$
\end {definition}
For a torus bundle $T_A$, we can isotope  $A$ to be a linear diffeomorphism, which
means that we have $A \in GL_2(\Z)$. $T_A$ is called a torus bundle with monodromy matrix $A$. In orientable case, $A$ must be in the special linear group $SL_2(\Z)$.

Recall that $SL_2(\Z)$ is generated by $S=\left(\begin{array}{cc}0&-1\\1&0\end{array}\right)$ and $T=\left(\begin{array}{cc}1&1\\0&1\end{array}\right)$. Denote by $I$ the unit matrix. Using that $S^2=-I=SIS$, we get that any $A$
can be decomposed as 
$A=S^\epsilon T^{a_1}ST^{a_2}\dots ST^{a_n}S^\eta$, with $\epsilon,\eta\in\{0,1\}$.
If $\epsilon$ or $\eta$ is $0$, then we may replace $S^0$ by  $SISISIS$. In all cases we can obtain a decomposition starting and finishing with $S$:
$$A=\left(\begin{array}{cc}a&b\\c&d\end{array}\right)=ST^{a_1}ST^{a_2}\dots ST^{a_n}S\ .$$ 
Then a surgery presentation for $T_A$, \mbox{$A\in SL_2(\Z)$}, is given in \cite[Theorem A.4]{Kirby}. 
With appropriate numbering of the $n+2$ components, the linking matrix $B_{\mathcal L}$ is as follows :
 $$B_{\mathcal{L}}=\left(
  \begin{array}{cccccccc}
  a_1 & -1 & 0 &\cdots& \cdots& 0& -1& 0 \\
 -1 & a_2& -1&\ddots& \ddots & 0& 0 & 0 \\
0 & \ddots & \ddots&\ddots &\ddots &  0& 0& 0 \\
0 & \ddots & \ddots&\ddots &\ddots &  0& 0& 0 \\
  \vdots  & & \ddots & -1 &a_{n-1}  &-1& 0& 0 \\
 0 &\cdots & &0& -1 &a_n& -1& 0\\
 -1 & 0&\cdots &0& 0 &-1& 0& 0\\
 0 &\cdots &\cdots& 0& 0 &0& 0& 0
  \end{array}
\right)$$ 

 We first describe $H^1(T_A,\Z_2)$.
 \begin{lemma} 
Let $T_A$  be an oriented torus bundle with monodromy matrix \mbox{$A=\left(\begin{array}{cc}
a & b \\ 
c & d
\end{array}\right) \in SL_2(\Z) $}.
\begin{enumerate}[a)]
\item If $a$ and $d$ do not have the same parity, then $H^1(T_A, \Z_2) \simeq \Z_2$.
\item If $a \equiv d \equiv 0$ mod $2$, then $b \equiv c \equiv 1$ mod $2$ and $H^1(T_A, \Z_2) \simeq \Z^2_2$.
\item If  $a \equiv d \equiv 1$ mod $2$ then $bc \equiv 0$ mod $2$, and if  $b \equiv 1$ mod $2$ or  $c \equiv 1$ mod $2$
then we have \mbox{$H^1(T_A, \Z_2) \simeq \Z^2_2$}.
\item If $A\equiv I \text{ mod } 2$, then $H^1(T_A, \Z_2) \simeq \Z^3_2$.
\end{enumerate}
 \end{lemma}
 \begin{proof}
 
 This can be proved directly from the monodromy matrix with a Mayer-Vietoris argument. We will rather use the linking matrix in order to have a description of generators in the surgery picture. 
 We  use successive  row transformations on $B_{\mathcal L}$ in order to clarify its rank and kernel, namely
  $R_{n-1} \leftarrow R_{n-1}+a_{n} R_{n}$,
 $R_i \leftarrow R_i+a_{i} R_{i+1}-R_{i+2}$, for $i=n-2$ to $1$,
  and then  $R_{n+1}\leftarrow R_{n+1}-R_2$.
We obtain a matrix $B'$ with same kernel mod $2$. 
$$B'_{\mathcal L}=\left(
  \begin{array}{ccccccccccccc}
    0 & 0 & 0  & \ldots & &\cdots& &0 & 0 &  \alpha_1&\gamma_{1}-1&0\\
    -1& 0 &0&\ldots & &  \cdots& & 0 & 0 &  \alpha_2&\gamma_{2}&0  \\
  0   & -1  & 0 &\ldots & & \cdots& &0 &  0 & \alpha_3&\gamma_3&0 \\
\vdots & \ddots &\ddots &\ddots & & &  &  &   & \vdots &\vdots& 0\\
\vdots &  & \ddots &\ddots &\ddots &&  &  &  &\vdots &\vdots& \\
0& \ldots  & &0  &-1 &0  &\cdots& 0 & 0  &  \alpha_{i}&\gamma_i& 0\\
    \vdots &  &  & & \ddots & \ddots&\ddots &  & & \vdots &\vdots& 0\\
     \vdots &  &  & &&\ddots  &\ddots  &\ddots  && \vdots &\vdots& 0 \\
    0 & \ldots &  &  \ldots& & & 0&  -1&   0& \alpha_{n-1} &\gamma_{n-1}&0\\
    0 & \ldots &  &  \ldots& & & 0&  0& -1 &  \alpha_n&\gamma_{n}& 0 \\
    0 & \ldots &  &  \ldots& & & 0&  0& 0 &-1-\alpha_2&-\gamma_2& 0 \\
0 & \ldots &  &  \ldots& & & 0&  0& 0 &0& 0& 0
  \end{array}
\right)
$$
with the decreasing recursive formulas  $\alpha_{n}=a_n$, $\gamma_n=-1  $, $\alpha_{n+1} =1$, $\gamma_{n+1}=0$ and $\alpha_{i}= a_i\alpha_{i+1}-\alpha_{i+2}$, $\gamma_{i}= a_i\gamma_{i+1}-\gamma_{i+2}$ for $i=n-1$ to $1$.

A decreasing recursion shows that for $i=n$ to $1$ we have
$$ST^{a_i}\dots ST^{a_n}S=\left(\begin{array}{cc}
-\alpha_{i+1} & -\gamma_{i+1} \\ 
\alpha_i & \gamma_i
\end{array}\right) $$
Hence we get
$$A=\left(\begin{array}{cc}
a & b \\ 
c & d
\end{array}\right) =\left(\begin{array}{cc}
-\alpha_{2} & -\gamma_{2} \\ 
\alpha_1 & \gamma_1
\end{array}\right) $$ 
We see that a solution \mbox{$x={}^t(x_1,\dots,x_{n+2})$} of the modulo $2$ kernel equation is uniquely determined by the last three components which are submitted to the
conditions
$$\left\lbrace
\begin{array}{lll}
\alpha_1x_n+(\gamma_1-1)x_{n+1}&=&0\\
(-1-\alpha_2)x_n-\gamma_2x_{n+1}&=&0
\end{array}
\right.$$
which is
$$\left\lbrace
\begin{array}{lll}
cx_n+(d-1)x_{n+1}&=&0\\
(a-1)x_n+bx_{n+1}&=&0
\end{array}
\right.$$
The description of the basis will use the elements in $\Z_2^{n+2}$ defined by $$u =\left(\begin{array}{c}0\\0\\\vdots\\0\\0\\0\\ 1\end{array}\right),
v =\left(\begin{array}{c}\gamma_2\\\gamma_3\\\vdots\\\gamma_n\\0\\1\\ 0\end{array}\right), w =\left(\begin{array}{c}\alpha_2\\\alpha_3\\\vdots\\\alpha_n\\1\\0\\ 0 \end{array}\right).$$ 
\begin{enumerate}[a)]
\item If  $a, d$ do not have the same parity then $$ H^1(T_A,\Z_2)\simeq \ker (B_\mathcal{L}\otimes \Z_2) =\Z_2,$$
 generated by $u$.
\item If $a, d$  are even then $b$ and $c$ are odd and $$H^1(T_A, \Z_2)\simeq \ker (B_\mathcal{L}\otimes \Z_2) \simeq \Z^2_2,$$ with basis $(u,v+w)$.
\item If $a, d$  are odd, then $bc$ is even. If moreover $b \equiv 1$ mod $2$ or  $c \equiv 1$ mod $2$, then $$  H^1(T_A, \Z_2)\simeq \ker (B_\mathcal{L}\otimes \Z_2) \simeq\Z^2_2.$$
 A basis is $(u,w)$ when $b$ is odd, and $(u,v)$ when $c$ is odd.
 \item If $a, d$  are odd and $b \equiv c \equiv 0$ mod $2$, then $$ H^1(T_A, \Z_2)  )\simeq \ker (B_\mathcal{L}\otimes \Z_2) \simeq\Z_2^3,$$ with basis $(u,v,w)$.
\end{enumerate}

 \end{proof}

 We will also need to know about the cokernel.
 \begin{lemma} 
Let $T_A$  be an oriented torus bundle with monodromy matrix \mbox{$A=\left(\begin{array}{cc}
a & b \\ 
c & d
\end{array}\right) \in SL_2(\Z)$} and surgery presentation given by a link $\mathcal{L}$ as above. Then $H^2(T_A,\Z)$ is isomorphic to $$\Z\oplus \mathrm{coker}(A-I).$$
If $A=ST^{a_1}ST^{a_2}\dots ST^{a_n}S$ is a decomposition giving a surgery link presentation $\mathcal{L}$ and
 $e_1,\dots,e_{n+2}$ is the canonical basis of the module $\Z^{n+2}$ on which the linking matrix $B_\mathcal{L}$ acts, then the free summand $\Z$ is generated by $e_{n+2}$, and 
$A-I$ is a linear map \mbox{$\Z^2\cong\mathrm{Span}(e_{n},e_{n+1}) \rightarrow \mathrm{Span}(e_{n+1},e_{1})$}.
 \end{lemma} 
 \begin{remark}
 If $a+d=2$ and $A\neq I$, then the matrix $A-I$ has rank one. Its image is a discrete subgroup in a line, hence a cyclic group.
 We denote by $\gcd(A-I)$ a generator of this group or equivalently which minimize the norm. It can be obtained from the coloms of $A-I$ by adapting the euclidean algorithm.
 \end{remark}

\begin{proof}
The cokernel of   $B_{\mathcal L}$ is equal to $\Z^{n+2}$ modulo the space generated by the colomns. Using colomns $2$ to $n$ we get that the cokernel is generated by the  basis vectors $e_1$, $e_{n+1}$, $e_{n+2}$.
 This implies that 
$$\mathrm{coker}(B_{\mathcal L})=E/\mathrm{Im}(B_{\mathcal L})\cap E\ ,$$
where $E=\mathrm{Span}(e_{n+2},e_{n+1},e_{1})$. 

A linear combination $y_1e_1+y_{n+1}e_{n+1}+y_{n+2}e_{n+2}$ belongs to $\mathrm{Im}(B_{\mathcal L})$ if and only if
the linear system $B_{\mathcal L}x=y$ has at least a solution; here $y_2=y_3=\dots=y_n=0$.
The  row transformations used for kernel produce an equivalent system $B'_{\mathcal L}x=y$. Note that the right hand of the equation does not change under the row transformations, because we modify with rows $n$ downto $2$ for which the component of $y$ vanishes.
We see that a solution is uniquely determined by the last three components which are submitted to the
conditions
$$\left\lbrace
\begin{array}{lll}
\alpha_1x_n+(\gamma_1-1)x_{n+1}&=&y_1\\
(-1-\alpha_2)x_n-\gamma_2x_{n+1}&=&y_{n+1}
\end{array}
\right.$$
which is
$$\left\lbrace
\begin{array}{lll}
(a-1)x_n+bx_{n+1}&=&y_{n+1}\\
cx_n+(d-1)x_{n+1}&=&y_1
\end{array}
\right.$$
The conclusion follows.
\end{proof}

 We are now able to discuss the Borsuk-Ulam index for all double covers of $T_A$. In the case $a+d=2$ the colomns of the matrix
 $A-I=\left(\begin{array}{cc}a-1&b\\c&d-1\end{array}\right)$ generate a cyclic group. 
 
\begin{theorem}
Let $T_A$  be an oriented torus bundle with monodromy matrix $$A=\left(\begin{array}{cc}
a & b \\ 
c & d
\end{array}\right) \in SL_2(\Z), $$ and let $(M, \tau)$ be the double covering of $T_A$ with deck transformation $\tau$ defined by the non zero characteristic class $x=\alpha u+\beta v+\gamma w$, $\alpha,\beta,\gamma\in \Z_2$ submitted to the kernel condition . The Borsuk-Ulam index
$ind_{\Z_2}(M, \tau)=f(A,x)$ is given as follows.
\begin{enumerate}
\item For every $A$ we have $f(A,u)=1$.
\item If $x=\beta v+\gamma w$ is a non zero characteristic class, then 
$f(A,u+x)=f(A,x)$.
\item All Borsuk-Ulam indices are given by following.
\begin{enumerate}
\item If $a$ and $d$ do not have the same parity, then only $u$ is characteristic and $f(A,u)=1$
\item If $a \equiv d \equiv 0$ mod $2$, then 
$$f(A, u) =1$$
$$f(A,v+w)=3 \Leftrightarrow f(A,u+v+w)=3 \Leftrightarrow (a+b)(c+d -2) \equiv 1 \;\;\text{mod} \;\; 4,$$
$$f(A,v+w)=1 \Leftrightarrow f(A,u+v+w)=1\Leftrightarrow
a+d=2\text{ and }\left(\begin{array}{cc}a+b-1\\c+d-1\end{array}\right)\in  2\gcd(A-I)\, \Z
.$$
Here $\gcd(A-I)$ is a generator of the cyclic group generated by the columns of $A-I$.
\item If  $a \equiv d \equiv 1$ mod $2,$ then
\begin{enumerate}
\item if $b \equiv 1$ mod $2$ we have
$$f(A, u) =1$$
$$f(A,w)=3 \Leftrightarrow f(A,u+w)=3 \Leftrightarrow  a(c-1) \equiv 1\;\;\text{mod} \;\; 4,$$
$$f(A,w)=1 \Leftrightarrow f(A,u+w)=1 \Leftrightarrow a+d=2\text{ and }\left(\begin{array}{cc}a-1\\c\end{array}\right)\in  2\gcd(A-I)\, \Z
.$$
\item if $c \equiv 1$ mod $2$ we have
$$f(A, u) =1$$
$$f(A, v)=3 \Leftrightarrow f(A, u+v)=3 \Leftrightarrow  b\equiv 2 \;\;\text{mod} \;\; 4,$$
$$f(A, v)=1 \Leftrightarrow f(A, u+v)=1 \Leftrightarrow a+d=2 \text{ and } \left(\begin{array}{cc}b\\d-1\end{array}\right)\in  2\gcd(A-I)\, \Z
..$$
\end{enumerate}
\item If $A\equiv I \text{ mod } 2$, there are $7$ possibilities for the classifying class $x$, and we have:
 $$f(A, u) =1$$
 $$f(A, v)=3 \Leftrightarrow f(A, u+v)=3 \Leftrightarrow  b\equiv 2 \;\;\text{mod} \;\; 4,$$
$$f(A, v)=1 \Leftrightarrow f(A, u+v)=1 \Leftrightarrow  a+d=2 \text{ and } \left(\begin{array}{cc}b\\d-1\end{array}\right)\in  2\gcd(A-I)\, \Z .$$
$$f(A,w)=3 \Leftrightarrow f(A,u+w)=3 \Leftrightarrow  a(c-1) \equiv 1 \;\;\text{mod} \;\; 4,$$
$$f(A,w)=1 \Leftrightarrow f(A,u+w)=1 \Leftrightarrow a+d=2\text{ and }\left(\begin{array}{cc}a-1\\c\end{array}\right)\in  2\gcd(A-I)\, \Z .$$
$$f(A,v+w)=3 \Leftrightarrow f(A,u+v+w)=3 \Leftrightarrow (a+b)(c+d -2) \equiv 1 \;\;\text{mod} \;\; 4,$$
$$f(A,v+w)=1 \Leftrightarrow f(A,u+v+w)=1\Leftrightarrow 
a+d=2\text{ and }\left(\begin{array}{cc}a+b-1\\c+d-1\end{array}\right)\in  2\gcd(A-I)\, \Z
.$$
\end{enumerate}
\end{enumerate}
\label{th11}
\end{theorem}

\begin{proof}

We have integral lifts for $u$, $v$, $w$:
$$ U =\left(\begin{array}{c}0\\0\\\vdots\\0\\0\\0\\ 1\end{array}\right),
V =\left(\begin{array}{c}\gamma_2\\\gamma_3\\\vdots\\\gamma_n\\0\\1\\ 0\end{array}\right),W=\left(\begin{array}{c}\alpha_2\\\alpha_3\\\vdots\\\alpha_n\\1\\0\\ 0 \end{array}\right).$$
We get $B_{\mathcal L}U=0$ which proves 1. and 2. Then we compute:
$$
B_{\mathcal L}V =\left(\begin{array}{c}\gamma_1-1\\0\\\vdots\\0\\-\gamma_2\\ 0\end{array}\right)=\left(\begin{array}{c}d-1\\0\\\vdots\\0\\b\\ 0\end{array}\right),
B_{\mathcal L}W=\left(\begin{array}{c}\alpha_1\\0\\\vdots\\0\\-\alpha_2-1\\ 0 \end{array}\right)=\left(\begin{array}{c}c\\0\\\vdots\\0\\a-1\\ 0 \end{array}\right).$$
\begin{enumerate}[(a)]
\item  The case $a\not\equiv d$ modulo $2$ is covered by 1.
\item Case where $a,d$ are even and hence $b,c$ are odd. The characteristic class will take the values $u, v+w, u+v+w$. The case $u$ is covered by 1). Furthermore, 
we have$$ B_{\mathcal L}(V+W) =\left(\begin{array}{c}c+d-1\\0\\\vdots\\0\\a+b-1\\ 0\end{array}\right)\ .$$ and $\transp(V+W) B_{\mathcal L} (V+W) = - [(a+b)(c+d -2) + 1]$. Now we use \ref{th8}.
The index $3$ criterion gives $f(A, v+w)=3$ if only if $(a+b)(c+d -2) + 1 \equiv 2 \;\;\text{mod} \;\; 4$. \\
For the index $1$ criterion, following the description of $\mathrm{coker}(B_\mathrm{L})$, we get that \mbox{$\frac{1}{2}B_{\mathcal L}(V+W)$} vanishes in the cokernel if and only if the system below has integral solution.
$$\left\lbrace\begin{array}{lcl}(a-1)x+by&=&\frac{a+b-1}{2}\\
cx+(d-1)y&=&\frac{c+d-1}{2}\end{array}
\right.$$
In the case $a+d\neq 2$ the system has a unique rational solution $x=y=\frac{1}{2} $ and no integral solution.
In the case $a+d=2$, the system has integral solution if and only if $\left(\begin{array}{c}\frac{a+b-1}{2}\\\frac{c+d-1}{2}\end{array}\right)$ belongs to the group generated by the columns of $(A-I)$ which is $\Z\gcd(A-I)$, whence the condition written in theorem.
\item If  $a \equiv d \equiv 1$ mod $2$ then we have two cases 
\begin{enumerate}
\item if $b \equiv 1$ mod $2$, the characteristic class will take the values $u$, $w$ or $u+w$. $f(A, u)$ is computed in 1). \\
We have $\transp W B_{\mathcal L} W = - ac+a-1$. 
The index $3$ criterion gives $f(A, w)=3$ if only if $ac-a+1 \equiv 2 \;\;\text{mod} \;\; 4$. \\
Furthermore, \mbox{$\frac{1}{2}B_{\mathcal L}W$} vanishes in the cokernel if and only if the system below has integral solution.
$$\left\lbrace\begin{array}{lcl}(a-1)x+by&=&\frac{a-1}{2}\\
cx+(d-1)y&=&\frac{c}{2}\end{array}
\right.$$
In the case $a+d\neq 2$ the system has a unique rational solution $x=\frac{1}{2} $, $y=0$, and no integral solution.
In the case $a+d=2$, the system has integral solution if and only if $\left(\begin{array}{c}\frac{a-1}{2}\\\frac{c}{2}\end{array}\right)$ belongs to  $\Z\gcd(A-I)$. 

\item if $c \equiv 1$ mod $2$, the characteristic class will take the values $u$, $v$ or $u+v$.  We have $\transp V B_{\mathcal L} V =b(-d+2)$.
Note that $d$ is odd and we get
$f(A, v)=3$ if only if $ b\equiv 2 \;\;\text{mod} \;\; 4.$\\
Furthermore, using the expression of $B_{\mathcal L}V$ above, we see that
$f(A, v)=1$ if only if $a+d=2$ and $\left(\begin{array}{c}\frac{b}{2}\\\frac{d-1}{2}\end{array}\right)$ belongs to  $\Z\gcd(A-I)$. 
\end{enumerate}
\item The proof in the case $A\equiv I \text{ mod } 2$ uses the same computations as previous cases.
\end{enumerate}
\end{proof}

\subsection{Application for 3-manifolds having surgery
 presentations by diagonal linking matrices}
In this part, let us consider the triples $(M, \tau, \R^n)$, and $N=M/\tau =S^3(\mathcal{L})$ i.e $N$ is obtained from $S^3$ by integral surgery on $\mathcal{L}$. Suppose that in this case the linking matrix $B_\mathcal{L}$ of $N$ is a $n\times n$ diagonal matrix of the form
\begin{equation*}
A=M(f)=
\left(
  \begin{array}{ccccccccc}
    a_1 & 0 & \ldots  & \ldots & &0 &0 & \ldots & 0\\
    0& \ddots && & &  \vdots &\vdots  & 0 & \vdots \\
  \vdots   &   & a_\nu & & & \vdots & \vdots &   & \vdots  \\
\vdots &  & & b_1  & &  \vdots& \vdots &   & \vdots  \\
\vdots &  &  & &\ddots & 0 &  &  &  \\
0& \ldots  & &\ldots  & &b_\mu & 0 & \ldots  & 0  \\
    0 & \ldots &  &  \ldots & & 0& 0 & \ldots & 0\\
    \vdots &  &  & & &  & \vdots & 0& \vdots  \\
    0 & \ldots &  &  \ldots& &  0&  0 & \ldots& 0\\
  \end{array}
\right)
\end{equation*}
where  $a_i, b_i$ in the first block are respectively even and odd integers, and all the others coefficients are zero.
In this case, we have $H^1(N_{\mathcal{L}}, \Z_2)=(\Z_2)^\nu \oplus (\Z_2)^m$, with $m=n-(\nu+v)$.

Then the classifying class is the form $x=x'+x''$ where $x' \in(\Z_2)^\nu$ and $x'' \in (\Z_2)^m$.  
\begin{proposition}
Under the above hypothesis, we have
\begin{enumerate}
\item $ind_{\Z_2}(M, \tau) \geq 2$ if and only if $x' \neq 0$.
\item $ind_{\Z_2}(M, \tau) =3$ if and only if   $\displaystyle \sum_{{1\leq i\leq \nu}\atop{x'_i\neq 0}} a_i  $ is not divisible by $4$.
\end{enumerate}
\label{diag}
\end{proposition}
\begin{proof} 

Denote by $X\in \Z^n$ (resp. $X'$,$X''$) the lifts of $x$ (resp. $x'$, $x''$) whose components are in $\{0,1\}$. 
We will use the criterions in Theorem \ref{th8}.
 We have $B_{\mathcal{L}}X=B_{\mathcal{L}}X'$.
The non zero components of $\frac{1}{2}B_{\mathcal{L}}X$ are $\frac{a_i}{2}$ for each non zero $x'_i$.
It vanishes if $x'=0$, and 
if $x'\neq 0$ then $\frac{1}{2}B_{\mathcal{L}}X$ does not belongs to the image of $B_{\mathcal{L}}$. This proves
 the first statement. 

The second statement follows from the computation
$$ \frac{1}{2} \transp X B_{\mathcal{L}} X =  \frac{1}{2} \transp X B_{\mathcal{L}} X = \frac{1}{2}\sum_{{1\leq i\leq \nu}\atop{x'_i\neq 0}} a_i \ .   $$
\end{proof}
\begin{remark}
 Proposition \ref{diag} can be used in more general cases. Indeed, let $(M, \tau)$ be a $\Z_2-$space, where $M$ is a compact connected oriented 3-manifold, $N=M/\tau$ is such that $N= N_\mathcal{L}$ where $\mathcal{L}$ is a framed link. Using Corollary 2.5 in \cite{Otsuki}, there exists lens spaces $L(p_i, 1)$, $i=\overline{1, \nu}$ with $|p_i|\leq | H_1(N, \Z) |$, such that the manifold $N'= N \# L(p_1, 1) \# L(p_2, 1) \# ......\# L(p_\nu, 1)$ can be obtained by integral surgery along some algebraically split framed link $\mathcal{L'}$. 
This allows to compute  the index $ind_{\Z_2}(M, \tau)$, by using $N'$ with characteristic class extended by zero on the lens spaces.
\end{remark}
\subsection{The Borsuk-Ulam theorem for $S^1\times S^2$}
In this section we will discuss all free involutions on $S^1\times S^2$.
All the applications given above deal with the cases where the orbit space $N$ is a compact connected oriented 3-manifold.
Here we will also get unoriented orbit spaces.  The following theorem, proved by Y.Tao in 1962, describes all free involutions.
\begin{theorem}\cite {Tao}
If $T$ is a fixed point free involution of $S^1\times S^2$, then the orbit space $N=M/T$ is  homeomorphic either to (1) $S^1\times S^2$ or (2) 3-dimensional Klein bottle (we denote it by $K^3 )$, or (3) $S^1\times \R P^2$, or (4) $\R P^3 \# \R P^3.$
\end{theorem}
Using the theorem above, we can state the following.
\begin{theorem}
Let  $\tau$ be a free $\Z_2$-action on $S^1\times S^2$, and let $N$ be the orbit space of this action. We have the following cases:
\begin{enumerate}
\item If $N$ is $S^1\times S^2$, then $ind_{\Z_2}(S^1\times S^2, \tau)=1.$
\item If $N$ is the 3-dimensional Klein bottle $K^3$, then $ind_{\Z_2}(S^1\times S^2, \tau)=1.$
\item If $N$ is $S^1\times \R P^2$, then $ind_{\Z_2}(S^1\times S^2, \tau)=2.$
\item If $N$ is $\R P^3 \# \R P^3$, then $ind_{\Z_2}(S^1\times S^2, \tau)=2.$
\end{enumerate}
\label{th 21}
\end{theorem}
\begin{proof}

\noindent \textbf{The first case : $N=S^1\times S^2$}\\
Here after calculations of the cohomology groups of the 3-manifold $S^1\times S^2$, we obtain $H^1(S^1\times S^2, \Z_2)\simeq \Z_2$ and $H^2(S^1\times S^2, \Z)\simeq \Z$, so the Bockstein homomorphism $\beta \in Hom(\Z_2, \Z)$ vanishes. Hence, $ind_{\Z_2}(S^1\times S^2, \tau)=1.$

\noindent\textbf{The second case : $N=K^3.$}\\
Here, $K^3 = [0, 1]\times S^2/(1, x)\sim(0, -x)$. We take $$A=[0, \frac{1}{2}]\times S^2, B=[\frac{1}{2}, 1]\times S^2 \: \: and \: \: A\cap B=\{\frac{1}{2}\}\times S^2 \sqcup \{0\}\times S^2.$$ Using the long exact sequence of Mayer-Vietoris and the universal coefficients theorem,  we obtain $H^1(K^3, \Z_2)\simeq \Z_2$ and $H^2(K^3, \Z) \simeq 0$. Therefore, the Bockstein homomorphism $\beta : H^1(K^3, \Z_2) \longrightarrow H^2(K^3, \Z)$ is the zero homomorphism, and in this case we have $ind_{\Z_2}(S^1\times S^2, \tau)=1.$

\noindent\textbf{The third case : $N=S^1\times \R P^2.$}\\
Using the K$\overset{..}u$nneth formula, we obtain $H^1(S^1\times \R P^2, \Z_2)\simeq \Z_2 \oplus \Z_2.$\\
The first factor of this direct sum is generated by $u_1\times v_0$, the image of $u_1 \otimes v_0$ by the cross product $\times$, where $u_1, v_0$ are respectively the generators of $H^1(S^1, \Z_2)\simeq \Z_2$  and $H^0(\R P^2, \Z_2)\simeq \Z_2.$ Since, the ring $\Z_2$ has an identity element and $v_0$ is only one no zero class in $H^0(\R P^2, \Z_2)$, it corresponds at the identity element of the cup product defined by the 0 cycle taking the value 1 on each singular 0 simplex. Denote it by  $\textbf{1}$. Hence, $u_1 \times\textbf{1} $ is the generator of the first factor of the direct sum in $H^1(S^1\times \R P^2, \Z_2)$. \\Using a similar argument, we have $ \textbf{1}\times v_1 $ is the generator of the second factor of the last direct sum, where $\textbf{1} \in H^0(S^1, \Z_2) \simeq \Z_2$ and $v_1\in H^1(\R P^2, \Z_2)\simeq \Z_2$ are the generators of these modules.\\
The classifying class $x\in H^1( S^1\times \R P^2, \Z_2)$ in this case will take the values $(0, 1)$ in $\Z_2 \oplus \Z_2$ since its associated covering space is $S^1\times S^2$. So using the properties of the cup product and the cross product, we have $$\beta_2(x)=x\smallsmile x=(0+\textbf{1}\times v_1)\smallsmile(0+\textbf{1}\times v_1) = (\textbf{1}\times v_1)\smallsmile (\textbf{1}\times v_1)=\textbf{1}\times v^2_1.$$  Since $v_1 $ is the generator of $H^1(\R P^2, \Z_2)$, we know that  $v^2_1$ is also the generator of $ H^2(\R P^2, \Z_2)$ , then $v^2_1 \neq 0$ and $\beta_2(x) \neq 0$. Thus, $ind_{\Z_2}(S^1\times S^2, \tau)\geq 2.$\\
No we calculate the triple cup, $$x^3= \beta_2 (x)\smallsmile x = (\textbf{1}\times v^2_1 )\smallsmile (\textbf{1}\times v_1) = \textbf{1}\times v^3_1 = 0,$$ because of dimension reasons. Then, in this case we have $ind_{\Z_2}(S^1\times S^2, \tau)=2$.\\
\textbf{The forth case : $N=\R P^3 \# \R P^3.$}\\
Here we are in the case of oriented manifolds.
 We can use a surgery presentation, namely a trivial link with $2$ components and framings $2$. 
 The linking matrix is the matrix $$B_\mathcal{L}=\left(
  \begin{array}{cc}
    2 & 0 \\
    0 & 2 \\
  \end{array}
\right).$$
 There are $3$ non equivalent connected double covers of $\R P^3\# \R P^3$. Those with classifying class $(1,0)$ and $(0,1)$ are homeomorphic to $\R P^3\#\R P^3$, the remaining one is
$S^1\times S^2$ with classifying class $(1,1)$.
From the discussion of diagonal surgery presentation in previous section, we get $ind_{\Z_2}(S^1\times S^2, \tau)= 2.$
\end{proof}

\subsection{Application to the 3-Klein bottle $K^3$}
The study of the previous application allows us to remark that there is a double covering  $p: K^3 \longrightarrow S^1\times \R P^2$ which associates to each class $ [(t, x)]\in K^3$, the class $[(t, [x]]\in S^1\times \R P^2$. The corresponding free $\Z_2-$action on $K^3$ associated to this covering is $[t,x]\rightarrow [t,-x]$. We have seen that $H^1(S^1\times \R P^2, \Z_2)\simeq \Z_2 \oplus \Z_2$ which implies that there are $3$ non equivalent connected double covers of $S^1\times \R P^2$ : 
$S^1\times S^2$ with classifying class $(0,1)$, $S^1\times \R P^2$ with classifying class $(1,0)$ and $K^3$ with classifying class $(1,1)$.

\begin{proposition}
Let $(K,\tau)$ be the $3$-dimensional Klein bottle \mbox{$K= [0, 1]\times S^2/(1, x)\sim(0, -x)$} with involution $\tau: [t,x]\rightarrow [t,-x]$, then we have $ind_{\Z_2}( K^3, \tau) =3.$
\label{prop 22}
\end{proposition}
\begin{proof}

First as in the previous subsection,
 $x=u_1 \times\textbf{1} + \textbf{1}\times v_1$. We calculate 
  \begin{eqnarray*}
\beta_2(x)&=& (u_1 \times\textbf{1} +\textbf{1}\times v_1) \smallsmile (u_1 \times \textbf{1}+\textbf{1}\times v_1)\\
&= &(u_1^2 \times\textbf{1})+ (u_1 \smallsmile \textbf{1}) \times (\textbf{1} \smallsmile v_1 ) - (\textbf{1}\smallsmile u_1)\times (v_1 \smallsmile\textbf{1})+(\textbf{1}\times v_1^2)\\
&= &u_1^2 \times\textbf{1}+ u_1\times v_1 - u_1\times v_1 +\textbf{1}\times v_1^2\\
& =&\textbf{1}\times v_1^2,
\end{eqnarray*}
because $u_1^2=0 \in H^2(S^1, \Z_2) \simeq 0.$
Since, $v_1^2 \in H^2(\R P^2, \Z_2)$ is no zero, so $\beta_2(x) \neq 0$ and $ind_{\Z_2}( K^3, \tau) \geq 2.$\\
For the triple cup, we have
\begin{eqnarray*}
x^3&=&\beta_2(x) \smallsmile x\\
&=&(\textbf{1}\times v_1^2 )\smallsmile (u_1 \times\textbf{1} + \textbf{1}\times v_1)\\
&=&(\textbf{1} \smallsmile u_1)\times (v_1^2 \smallsmile  \textbf{1}) + (\textbf{1}\times v_1^3)\\
&=&u_1 \times v_1^2 \neq 0
\end{eqnarray*}
also since $v_1^3=0 .$
\end{proof}

\end{document}